\documentclass[bibliography=totoc, 12pt]{amsart}
\usepackage[utf8]{inputenc}
\usepackage[margin=1in]{geometry} 
\usepackage{amsmath,amsthm,amssymb,amsfonts, dsfont, graphicx,mathtools,mathrsfs,tikz,hyperref,physics, stmaryrd, bm, tikz-cd,tkz-euclide,comment, enumerate, appendix, todonotes}
\makeatletter
\NewCommandCopy\@@pmod\pmod
\DeclareRobustCommand{\pmod}{\@ifstar\@pmods\@@pmod}
\def\@pmods#1{\mkern4mu({\operator@font mod}\mkern 6mu#1)}
\makeatother

\usepackage{subcaption}
\usepackage[backend=biber,bibencoding=utf8,style=numeric,url=true,
    isbn=false,     
    doi=false,           
    abbreviate=true,            
    giveninits=true, 
    natbib=true,
    hyperref=true]{biblatex}
\usetikzlibrary{decorations.pathreplacing,decorations.markings}

\theoremstyle{plain}
\newtheorem{theorem}{Theorem} 

\newtheorem{lemma}{Lemma}
\newtheorem{conjecture}{Conjecture}
\newtheorem{mainthm}{Theorem}

\numberwithin{equation}{section}

\theoremstyle{definition}

\theoremstyle{remark}

\title[The singular series of a cubic form in many variables]{The singular series of a cubic form in many variables and a new proof of Davenport's Shrinking Lemma}
\author[C. Bernert]{Christian Bernert}
\address{Mathematisches Institut, Bunsenstraße 3-5, 37073 Göttingen, Germany}
\email{christian.bernert@mathematik.uni-goettingen.de}

\begin{document}
\maketitle

\begin{abstract}
    We study the singular series associated to a cubic form with integer coefficients. If the number of variables is at least $10$, we prove the absolute convergence (and hence positivity) under the assumption of Davenport's Geometric Condition, improving on a result of Heath-Brown. For the case of $9$ variables, we give a conditional treatment. We also provide a new short and elementary proof of Davenport's Shrinking Lemma which has been a crucial tool in previous literature on this and related problems.
\end{abstract}

\section{Introduction}

Let $C(x_1,\dots,x_n) \in \mathbb{Z}[x_1,\dots,x_n]$ be a cubic form. We are interested in the existence of nontrivial integer solutions, i.e. nonzero vectors $\mathbf{x} \in \mathbb{Z}^n$ with $C(\mathbf{x})=0$.

Davenport \cite{davenport63} proved that if $n \ge 16$, such nontrivial solutions always exist. This remained the state of the art for almost half a century until Heath-Brown \cite{heath2007cubic} could extend the admissible range to $n \ge 14$, this has not been improved to date. Given that $10$ variables suffice to guarantee local solubility \cite{lewis_p_adic_zeroes}, it is generally expected that the same result should hold already when $n \ge 10$.

The Hardy-Littlewood Circle Method aims to prove the existence of solutions by proving that there are indeed many. If it works, it provides us with an asymptotic formula of the shape
\begin{equation}
    \label{asymp}
    \#\{\mathbf{x} \in \mathbb{Z}^n, x \in P\mathcal{B}\} =(1+o(1)) \cdot \mathfrak{I} \cdot \mathfrak{S} \cdot P^{n-3}
\end{equation}
as $P \to \infty$. Here $\mathcal{B} \subset \mathbb{R}^n$ is a suitably chosen box and $\mathfrak{I}$ and $\mathfrak{S}$ denote the usual \textit{singular integral} and the \textit{singular series} of the cubic form $C$, respectively, measuring the local solubility of $C$ over the fields $\mathbb{R}$ and $\mathbb{Q}_p$ for all primes $p$. The singular integral is rather unimportant for this paper, so we refer the reader to \cite{davenport63} for its precise definition and only mention that it is known to be positive for a suitable choice of $\mathcal{B}$ as soon as $n \ge 4$. The singular series is the key object of the present paper and will be defined and discussed in more detail in the next section. 

For now, let us continue discussing the heuristic asymptotic formula \eqref{asymp} and let us note that it clearly fails in certain degenerate situations. Indeed, when $C$ is reducible, it is easy to see that the count on the left-hand side is already $\gg P^{n-1}$. More generally, if our cubic form is of the shape $C(\mathbf{x})=x_1Q_1(\mathbf{x})+x_2Q_2(\mathbf{x})$ for certain quadratic forms $Q_1$ and $Q_2$, we still have $\gg P^{n-2}$ solutions and hence too many for $\eqref{asymp}$ to possibly hold.

The ingenious idea of Davenport to circumvent this fundamental problem was to establish a certain dichotomy: If the circle method fails to produce the asymptotic \eqref{asymp}, then this failure could be turned into an alternative proof of the existence of solutions, though not in such a precise quantitative manner.

To describe Davenport's idea in more detail, we write $C(\mathbf{x})=\sum_{i,j,k} c_{ijk} x_ix_jx_k$ where we assume the $c_{ijk}$ to be symmetric and integers (as we may by multiplying $C$ by $6$ if necessary). We then define the bilinear forms
\[
    B_i(\mathbf{x},\mathbf{y})=\sum_{j,k=1}^n c_{ijk} x_jy_k
\]
and the matrix $M(\mathbf{x})$ with entries
\[
    M(\mathbf{x})_{jk}=\sum_{i=1}^n c_{ijk} x_i
\]
so that $M(\mathbf{x})\mathbf{y}$ is the vector with entries $B_i(\mathbf{x},\mathbf{y})$. For later use we let $D(\mathbf{x})=\det M(\mathbf{x})$ and $r(\mathbf{x})=\text{rk} M(\mathbf{x})$. For a prime $p$, we will also need to consider the $\mathbb{F}_p$-rank of $M(\mathbf{x})$ which we denote by $r_p(\mathbf{x})$.

Let us now say that $C$ satisfies \textit{Davenport's Geometric Condition} if
\begin{equation}
\label{gc}
    \#\{\mathbf{x} \in \mathbb{Z}^n: \|x\|_{\infty} \le P, r(\mathbf{x})=r\} \ll P^{r+\varepsilon}
\end{equation}
is satisfied for all integers $r$ with $0 \le r \le n$.

We can then describe Davenport's result more concisely as follows:

\begin{mainthm} \label{A}If $C$ does not satisfy Davenport's Geometric Condition \eqref{gc}, then the equation $C(\mathbf{x})=0$ has a non-trivial integer solution.\end{mainthm}

\begin{mainthm} If $C$ satisfies Davenport's Geometric Condition \eqref{gc}, then the asymptotic formula $\eqref{asymp}$ holds with $\mathfrak{I}, \mathfrak{S}>0$ as soon as $n \ge 16$. In particular, there are non-trivial integer solutions  to $C(\mathbf{x})=0$.\end{mainthm}

Note that Theorem \ref{A} does not make any assumption on the number of variables $n$. This means that in trying to improve on the constraint on the number of variables, we are free to assume that the Geometric Condition is satisfied.

Indeed, this is what Heath-Brown did, showing

\begin{mainthm} If $C$ satisfies Davenport's Geometric Condition \eqref{gc}, then the asymptotic formula $\eqref{asymp}$ holds with $\mathfrak{I}, \mathfrak{S}>0$ as soon as $n \ge 14$. In particular, there are non-trivial integer solutions to $C(\mathbf{x})=0$.\end{mainthm}

In view of the above discussion, it is natural to conjecture that this should extend to $n \ge 10$:

\begin{conjecture}
\label{mainconj}
If $C$ satisfies Davenport's Geometric Condition \eqref{gc}, then the asymptotic formula $\eqref{asymp}$ holds with $\mathfrak{I}, \mathfrak{S}>0$ as soon as $n \ge 10$. In particular, there are non-trivial integer solutions to $C(\mathbf{x})=0$.\end{conjecture}


\section{Main results}

We now describe our main results. To this end, we need to return to the singular series $\mathfrak{S}$. It is defined in terms of the \textit{Gauß sums}
\[
    S(q,a)=\sum_{\mathbf{x} \pmod{q}} e\left(\frac{aC(\mathbf{x})}{q}\right)
\]
via
\[
    \mathfrak{S}=\sum_{q=1}^{\infty} \sum_{(a;q)=1} \frac{S(q,a)}{q^n}.
\]
By standard multiplicativity properties of the Gauß sums, this can (at least formally) also be written as an Euler product
\[
\mathfrak{S}=\prod_p \chi_p
\]
over all primes $p$ where
\[
    \chi_p=\sum_{k=0}^{\infty} \sum_{(a;p^k)=1} \frac{S(p^k,a)}{p^{kn}}
\]
is known as the $p$-adic density. By classical arguments it follows that $\chi_p>0$ if and only if $C(\mathbf{x})=0$ has a non-trivial solution over $\mathbb{Q}_p$. In particular, from \cite{lewis_p_adic_zeroes} we conclude that $\chi_p>0$ for all $p$ whenever $n \ge 10$.

So far we have ignored all convergence issues. The rearrangement between the series and the product representation of $\mathfrak{S}$ is only valid when either of the two is known to be absolutely convergent. Proving absolute convergence of $\mathfrak{S}$ is therefore crucial for switching between the two representations and also to conclude its positivity from the positivity of all individual factors $\chi_p$. Only then, the formula \eqref{asymp} truly captures the expected Local-Global Principle.

Previously, the absolute convergence for $\mathfrak{S}$ under the assumption of Davenport's Geometric Condition \eqref{gc} was known for $n \ge 11$ by work of Heath-Brown \cite{heath2007cubic}.

We begin by giving a new short and self-contained proof of this result. This new method then allows us to improve on previous work and establish the following.

\begin{theorem}
\label{thm1}
Assume that $n \ge 10$ and that $C$ satisfies Davenport's geometric condition. Then the singular series $\mathfrak{S}$ is absolutely convergent. In particular, $\mathfrak{S}>0$.
\end{theorem}

This can be seen as giving further evidence to Conjecture \ref{mainconj}. Moreover, the Gauß sums featuring in the definition of the singular series are closely related to the Weyl sums that would appear in a circle method proof of $\eqref{asymp}$. It is therefore to be hoped that the study of the Gauß sums and hence of the singular series can serve as a good model problem for our understanding of the more difficult Circle Method Problem.

We can also say something about the case $n=9$. We begin by proving that the only possible obstructions to absolute convergence are the Gauß sums with prime moduli. To deal with them, we then propose the following conjecture:

\begin{conjecture}
\label{conj}
Assume that $C$ satisfied the Geometric Condition \eqref{gc}. Then for all $n$ and uniformly in $1 \le H \le R$, we have
\[
    \#\{\mathbf{h} \le H, R<p \le 2R: r_p(\mathbf{h}) \le r\} \ll H^r \cdot R^{1+\varepsilon}.
\]
\end{conjecture}

We are able to prove the following:

\begin{theorem}
\label{thm2}
Under the assumption of Conjecture \ref{conj}, the singular series is absolutely convergent for $n=9$.  
\end{theorem}

In the last section, we return to the work of Davenport and Heath-Brown and give a short and elementary proof of \textit{Davenport's Shrinking Lemma}, which  is a crucial ingredient in the circle method approach to the cubic forms problem as pioneered by Davenport. The only previous proof of the Shrinking Lemma is due to Davenport and uses rather intricate tools from the geometry of numbers.

\subsection{Notation}

We use the usual notation $\mathcal{O}(\dots)$ and $\ll$ where the implicit constants are always allowed to depend on the cubic form $C(\mathbf{x})$. Moreover, whenever a bound involves $\varepsilon$, it means that the bound is true for all sufficiently small $\varepsilon>0$, but the implicit constant is allowed to depend on $\varepsilon$.

Moreover, we use the notation $e(x)=e^{2\pi ix}$ and $\|x\|=\min_{n \in \mathbb{Z}} \vert x-n\vert$. Whenever we write something like $\sum_{\mathbf{h}}$, the sum is restricted to integer vectors $\mathbf{h}$ and the given restrictions on the summation are to be read component-wise.

Finally, the condition $r \sim R$ denotes a restriction of $r$ to a dyadic interval $(R,2R]$.

\section{Review of previous bounds for $S(q,a)$}

The following simple lemma is good enough to recover all results previously obtained:

\begin{lemma}
\label{lemma1}
Let $q$ and $n$ be positive integers and let $M$ be a $n \times n$ matrix with integer coefficients. Then the size of the kernel of $M$ viewed as a map from $(\mathbb{Z}/q\mathbb{Z})^n$ to itself divides $\det M$.
In particular, if $q$ is a prime and $M$ has $\mathbb{F}_q$-rank at most $n-r$, then $p^r \mid \det M$.
\begin{proof}
Without loss of generality (that is, up to multiplication from both sides by invertible matrices), we may assume that $M$ is in Smith Normal Form with diagonal entries $a_1,\dots,a_n$. Then the kernel has size $\prod_{i=1}^n (a_i;q)$ which divides $\det M=\prod_{i=1}^n a_i$.
\end{proof}
\end{lemma}

We now recall the classical van der Corput differencing:
\begin{lemma}[Initial van der Corput Bound]
\label{lemma2}
Let $H \ge 1$ be arbitrary. Then, in the above notation, we have
\begin{equation}\label{vdc}
    \left(\frac{S(q,a)}{q^n}\right)^2 \ll \frac{1}{H^n} \sum_{1 \le \mathbf{h} \le H} \sqrt{\frac{1}{q^n}\#\{\mathbf{y} \pmod*{q}: q \mid B_i(\mathbf{y}, \mathbf{h})\}}.
\end{equation}
\end{lemma}
\begin{proof}
We set out by applying Cauchy-Schwarz to the identity
\[
    S(q,a)=\frac{1}{H^n}\sum_{\mathbf{x}\pmod*{q}} \sum_{1 \le \mathbf{h} \le H} e\left(\frac{aC(\mathbf{\mathbf{x}+\mathbf{h}})}{q}\right)
\]
to obtain after some manipulations
\[
\vert S(q,a)\vert^2 \ll \frac{q^n}{H^n} \sum_{-H \le \mathbf{h} \le H} \left\vert\sum_{\mathbf{x} \pmod*{q}} e\left(\frac{a\left(C(\mathbf{x}+\mathbf{h})-C(\mathbf{x})\right)}{q}\right)\right\vert.
\]
The lemma now follows by noting that the square of the absolute value of the inner sum is just
\[
\sum_{\mathbf{x},\mathbf{y}} e\left(\frac{a\left(C(\mathbf{x}+\mathbf{y}+\mathbf{h})-C(\mathbf{x}+\mathbf{y})-C(\mathbf{x}+\mathbf{h})+C(\mathbf{x}\right)}{q}\right)=\sum_{\mathbf{x},\mathbf{y}} e\left(\frac{a\sum_i x_iB_i(\mathbf{y},\mathbf{h})}{q}\right)
\]
and using orthogonality.
\end{proof}

Next, from Lemma \ref{lemma1} we see that $q^{r(\mathbf{h})-n}\#\{\mathbf{y} \pmod*{q}: q \mid B_i(\mathbf{y}, \mathbf{h})\}$ divides a non-zero $r(\mathbf{h}) \times r(\mathbf{h})$ minor of $M$ so that in particular
\[
\frac{1}{q^n}\#\{\mathbf{y} \pmod*{q}: q \mid B_i(\mathbf{y}, \mathbf{h})\} \ll \left(\frac{H}{q}\right)^{r(\mathbf{h})}. 
\]
Inserting this into Lemma \ref{lemma2} and using the geometric condition \eqref{gc}, we find that
\[
    \left(\frac{S(q,a)}{q^n}\right)^2 \ll \frac{1}{H^n} \sum_{-H \le \mathbf{h} \le H} \left(\frac{H}{q}\right)^{r(\mathbf{h})/2} \ll \frac{q^{\varepsilon}}{H^n} \sum_{r=0}^n \left(\frac{H^3}{q}\right)^{r/2}\ll q^{\varepsilon}\left(\frac{1}{H^n}+\frac{H^{n/2}}{q^{n/2}}\right)
\]
and putting $H=q^{1/3}$, we recover Heath-Brown's pointwise bound $S(q,a) \ll q^{5n/6+\varepsilon}$.

Recalling the definition of the $p$-adic factor in the product expansion of $\mathfrak{S}$, we now find that
\begin{align*}
    \chi_p&=\sum_{k=0}^{\infty} \sum_{(a;p^k)=1} \frac{S(p^k,a)}{p^{nk}}\\
    &=1+\sum_{(a;p)=1} \frac{S(p,a)}{p^n}+\mathcal{O}\left(\sum_{k=2}^{\infty} p^{k(1-n/6)+\varepsilon}\right)\\
    &=1+\sum_{(a;p)=1} \frac{S(p,a)}{p^n}+\mathcal{O}\left(p^{2-n/3+\varepsilon}\right)
\end{align*}
so that the estimation of the terms with $k \ge 2$ is satisfactory for the question of absolute convergence of $\mathfrak{S}$ as soon as $n>9$.

To establish Theorem \ref{thm1}, it therefore remains to show that
\[
    \sum_{p} \sum_{(a;p)=1} \frac{S(p,a)}{p^n}
\]
converges absolutely for $n \ge 10$. It would therefore clearly suffice to show that
\[
\sum_{p \sim R} \max_{(a;p)=1} \left\vert \frac{S(p,a)}{p^n}\right\vert \ll R^{-1-\delta}
\]
for which, by Cauchy-Schwarz, it suffices to establish
\[
    \sum_{p \sim R} \max_{(a;p)=1} \left\vert \frac{S(p,a)}{p^n}\right\vert^2 \ll R^{-3-\delta}
\]
for all choices of $R \ge 1$.

Using our previous line of argument, the LHS is bounded by
\[
    \frac{1}{H^n} \sum_{-H \le \mathbf{h} \le H} \sum_{R \le p<2R} \sqrt{\frac{1}{p^n}\#\{\mathbf{y} \pmod*{p}: p \mid B_i(\mathbf{y}, \mathbf{h})\}}=\frac{1}{H^n} \sum_{-H \le \mathbf{h} \le H} \sum_{R \le p<2R} p^{-r_p(\mathbf{h})/2}
\]
where $r_p(\mathbf{h})$ is the $\mathbb{F}_p$-rank of $M(\mathbf{h})$.
Heath-Brown's idea is now to distinguish two cases:

Those pairs $(\mathbf{h},p)$ with $r_p(\mathbf{h})=r(\mathbf{h})$ give a contribution bounded by
\begin{equation}
\label{fall1}
    \frac{1}{H^n} \sum_{-H \le \mathbf{h} \le H} \sum_{R \le p<2R} p^{-r(\mathbf{h})/2} \ll \frac{1}{H^n}\sum_{r=0}^n H^{r+\varepsilon} R^{1-r/2} \ll H^{\varepsilon}\left(\frac{R}{H^n}+\frac{R}{R^{n/2}}\right)
\end{equation}
by the Geometric Condition \eqref{gc}. The last term is satisfactory for $n>8$.

On the other hand, we need to estimate the contribution from those pairs $(\mathbf{h},p)$ with $r_p(\mathbf{h})<r(\mathbf{h})$. Here we use the implication from Lemma \ref{lemma1} that $p^{r(\mathbf{h})-r_p(\mathbf{h})}$ must divide a non-zero $r(\mathbf{h}) \times r(\mathbf{h})$-minor of $M(\mathbf{h})$ and is hence $\mathcal{O}(H^{r(\mathbf{h})})$ so that
\begin{equation}
\label{fall2.1}
    p^{-r_p(\mathbf{h})} \ll \left(\frac{H}{p}\right)^{r(\mathbf{h})}.
\end{equation}
Moreover, $p$ being a divisor of such a minor, there are at most $H^{\varepsilon}$ choices of such $p$ for any fixed $\mathbf{h}$. The total contribution of such pairs $(\mathbf{h},p)$ can therefore be bounded by
\begin{equation}
\label{fall2}
    \frac{H^{\varepsilon}}{H^n} \sum_{-H \le \mathbf{h} \le H} \left(\frac{H}{R}\right)^{r(\mathbf{h})/2} \ll H^{\varepsilon-n} \sum_{r=0}^n \left(\frac{H}{R}\right)^{r/2} \ll H^{\varepsilon}\left(\frac{1}{H^n}+\left(\frac{H^3}{R}\right)^{n/2}\right)
\end{equation}
again using the geometric condition.

Comparing the contributions from \eqref{fall1} and \eqref{fall2} we find that the optimal choice is $H=R^{\frac{n+2}{3n}}$ leading to the bound $R^{-(n-1)/3+\varepsilon}$ which is satisfactory when $n>10$.

\section{The case of ten variables}

When $n=10$, we observe that $H=R^{2/5+\delta}$ for sufficiently small $\delta>0$ leads to a satisfatory contribution from \eqref{fall1} and from all terms in \eqref{fall2} except when $r=n=10$. Moreover, even for this term it suffices to save another small power of $R$, which we do in \eqref{fall2.1} unless $r_p(\mathbf{h})=6$.
It therefore suffices to show that
\[
    \#\{\mathbf{h} \le H, p \sim R: r(\mathbf{h})=10, r_p(\mathbf{h})=6\} \ll H^{10-\delta'}
\]
for some $\delta'>0$ whenever $H=R^{2/5+\delta}$ for sufficiently small $\delta>0$.

To prove this, we use an argument inspired by a trick of Davenport \cite{davenport63} which he used to go from $17$ to $16$ variables. However, the presence of the extra averaging over $p$ requires a new idea.

By Lemma \ref{lemma1}, we have $p^4 \mid D( \mathbf{h})$ for all vectors $\mathbf{h}$ in question. Moreover, there are $p^4$ vectors $\mathbf{y} \in \{0,1,2,\dots,p-1\}^n$ with $p \mid B_i(\mathbf{y},\mathbf{h})$ for all $i$.

By the Pigeonhole principle, two of them differ by $\mathcal{O}(p^{3/5})$ in each component and by linearity of the $B_i$, this means that for each such $\mathbf{h}$ we get one solution $\mathbf{y}=\mathbf{y}(\mathbf{h}) \ne 0$ with $\|\mathbf{y}\|_{\infty} \ll p^{3/5}$ and $p \mid B_i(\mathbf{y},\mathbf{h})$.

Writing $B_i(\mathbf{y},\mathbf{h})=pm_i$, we find that $m_i=m_i(\mathbf{h}) \ll R^{\delta}$. Moreover, not all $m_i$ are zero since we assumed $r(\mathbf{h})=10$.

We can now count the number of pairs $(\mathbf{h},p)$ in question as follows: There are $\ll R^{10\delta}$ possible choices of the $m_i$.
For a fixed choice of $(m_1,\dots,m_n)$, we then study the number of possible choices of $(\mathbf{h},p)$. The general solution of the system $B_i(\mathbf{y},\mathbf{h})=pm_i$ is given by
\[
    y_j=p \cdot \frac{\sum_k m_k E_{jk}(\mathbf{h})}{D(\mathbf{h})}
\]
where the $E_{jk}$ are certain $9 \times 9$ minors of $M(\mathbf{h})$, in particular homogeneous forms of degree $9$ in $\mathbf{h}$.

Now certainly, for our given choice of the $m_i$, there is one $j$ such that the degree-$9$ form $E(\mathbf{h}):=\sum_k m_k E_{j,k}(\mathbf{h})$ is not identically zero. We conclude that $D(\mathbf{h}) \mid p \cdot E(\mathbf{h})$.

Let $G$ be the greatest common divisor of $D$ and $E$ and write $D=GD'$ and $E=GE'$ so that $D'(\mathbf{h}) \mid p \cdot E'(\mathbf{h})$ and $D'$ is coprime to $E'$. We thus find by Bézout's Theorem a non-zero linear combination $F$ of $D'$ and $E'$ that depends only on $h_2,\dots,h_n$. Hence $D'(\mathbf{h}) \mid p \cdot F(h_2,\dots,h_n)$. Note that the coefficients of all the polynomials depend on the $m_i$, but are all polynomially bounded in terms of $R$ which is sufficient for our application.

Now there are $\mathcal{O}(H^9)$ values of $\mathbf{h}$ where $F$ is zero and then $p$ as a divisor of $D(\mathbf{h})$ is determined up to $H^{\varepsilon}$ many choices, leading to a total bound of $H^{9+\varepsilon}$ for the number of pairs $(\mathbf{h},p)$ in this case.

On the other hand, if $F(h_2,\dots,h_{10})$ is non-zero, we see that $p \mid F(h_2,\dots,h_{10})$ by the following ad-hoc bootstrapping argument: Since $p^4 \mid D(\mathbf{h})=G(\mathbf{h}) \cdot D'(\mathbf{h})$ and $\deg G \le 9$ we have $G(\mathbf{h}) \ll H^9<p^4$ if $\delta>0$ is sufficiently small. Hence $p \mid D'(\mathbf{h})$. But if $\delta$ is small, this forces $\deg D' \ge 3$ and hence $\deg G \le 7$ so that $G(\mathbf{h}) \ll H^7<p^3$, again if $\delta$ is small. Hence $p^2 \mid D'(\mathbf{h})$ and hence $p \mid F(h_2,\dots,h_{10})$ as desired.

Finally, for any choice of $h_2,\dots,h_{10}$ with $F(\mathbf{h}) \ne 0$, this determines $p$ and $D'(\mathbf{h})$ up to $H^{\varepsilon}$ many choices and then also $h_1$ is determined up to finitely many choices, unless we are in a proper Zariski-closed subset of $h_2,\dots,h_{10}$. In any case, the total number of pairs $(\mathbf{h},p)$ can be bounded by $H^{9+\varepsilon}$. Summing up, we have thus shown that
\[
    \#\{\mathbf{h} \le H, p \sim R: r(\mathbf{h})=10, r_p(\mathbf{h})=6\} \ll R^{10\delta} \cdot H^{9+\varepsilon}
\]
which is satisfactory for $\delta>0$ sufficiently small. This finishes the proof of Theorem \ref{thm1}.

\section{The case of nine variables}

We now set out to discuss the case $n=9$, aiming for a proof of Theorem \ref{thm2}. To begin with, we need to discuss the case of higher prime powers. The contribution to $\chi_p$ of $S(p^k,a)$ for $k \ge 3$ is seen to be satisfactory even for $n=9$. For the contribution of the terms with $k=2$, our pointwise bound $S(p^2,a) \ll p^{5n/3+\varepsilon}$ just fails to be good enough when $n=9$.

However, we can use the averaging trick introduced in the previous section to also improve on this bound and therefore reduce the problem of absolute convergence of $\mathfrak{S}$ for $n=9$ to the study of $S(p,a)$ for primes $p$:

\begin{lemma}
\label{lemma}
For $n=9$, the sum
\[
    \sum_p \sum_{(a;p^2)=1} \frac{S(p^2,a)}{p^{2n}}
\]
is absolutely convergent. In particular, the singular series for $n=9$ converges absolutely if and only if
\[
    \sum_p \sum_{(a;p)=1} \frac{S(p,a)}{p^n}
\]
is absolutely convergent.
\end{lemma}
\begin{proof}
As before, a dyadic decomposition and an application of Cauchy-Schwarz reduce the problem to showing that
\[
    \sum_{p \sim R} \max_{(a;p^2)=1} \left\vert \frac{S(p^2,a)}{p^{2n}}\right\vert^2 \ll R^{-5-\delta}.
\]
From Lemma \ref{lemma2}, we see that the LHS is bounded by
\[
    \frac{1}{H^n} \sum_{-H \le \mathbf{h} \le H} \sum_{p \sim R} \sqrt{\frac{1}{p^{2n}} \#\{\mathbf{y} \pmod*{p^2}: p^2 \mid B_i(\mathbf{y},\mathbf{h})}\}.
\]
We continue by separating the cases $r_p(\mathbf{h})=r(\mathbf{h})$ and $r_p(\mathbf{h})<r(\mathbf{h})$. In the first case, the expression under the root is $(p^2)^{-r(\mathbf{h})}$ and using the geometric condition \eqref{gc} we obtain a contribution bounded by
\[
    \frac{RH^{\varepsilon}}{H^n} \sum_{r=0}^n \frac{H^r}{(R^2)^{r/2}} \ll R^{1+\varepsilon} \left(\frac{1}{H^n}+\frac{1}{R^n}\right)
\]
(compare this with \eqref{fall1}). In the second case, for each $\mathbf{h}$, there are at most $R^{\varepsilon}$ choices of $p$ and for each such pair the expression under the root is bounded by $\left(\frac{H}{p^2}\right)^{r(\mathbf{h})/2}$ so that the contribution in this case can be bounded by
\[
    \frac{R^{\varepsilon}}{H^n} \sum_{r=0}^n H^r \left(\frac{H}{R^2}\right)^{r/2} \ll R^{\varepsilon} \cdot \left(\frac{1}{H^n}+\left(\frac{H}{R^2}\right)^{n/2}\right)
\]
(compare this with \eqref{fall2}) and choosing $H=(R^2)^{\frac{n+1}{3n}}$ we end up with the total contribution of $\ll R^{1-\frac{2(n+1)}{3}}$ from both cases together, which is satisfactory as soon as $n>8$.
\end{proof}

We are now ready to prove Theorem \ref{thm2}:

\begin{proof}[Proof of Theorem \ref{thm2}]
By Lemma \ref{lemma} and the arguments from the previous discussion, it suffices to prove that
\[
    \sum_{p \sim R} \max_{(a;p)=1} \left\vert \frac{S(p,a)}{p^n}\right\vert^2 \ll R^{-1-\delta}
\]
for all choices of $R \ge 1$. Using Lemma \ref{lemma2}, the LHS is bounded by
\[
\frac{1}{H^n} \sum_{-H \le \mathbf{h} \le H} \sum_{p \sim R} p^{-r_p(\mathbf{h})/2} \ll \frac{1}{H^n} \sum_{r=0}^n R^{-r/2} \#\{\mathbf{h} \le H, p \sim R: r_p(\mathbf{h})=r\}.
\]
Assuming Conjecture \ref{conj}, this can be further estimated as
\[
    \ll \frac{1}{H^n} \sum_{r=0}^n R^{-r/2} H^rR^{1+\varepsilon} \ll \frac{R^{1+\varepsilon}}{H^n}+\frac{R^{1+\varepsilon}}{R^{n/2}}.
\]
Choosing e.g. $H=R^{1/2}$ we see that this is satisfactory as soon as $n>8$.
\end{proof}

Indeed, as can be seen from the above proof, only something weaker than Conjecture \ref{conj} is actually required. However, we do believe that this is the \lq right\rq{} way to put the conjecture, as the proposed upper bound is exactly the contribution that we a priori get from the terms with $r(\mathbf{h})=r$ and $p$ arbitrary, using the geometric condition \eqref{gc}.

We close this section by a few more remarks regarding Conjecture \ref{conj}. To start with, the cases $r=0$ and $r=n$ are easy to establish. Moreover, we can also prove the case $r=n-1$: Those $\mathbf{h}$ with $r(\mathbf{h})=n-1$ produce a satisfactory contribution by the geometric condition \eqref{gc}, as explained above. On the other hand, there can be only $\mathcal{O}(H^{n+\varepsilon})$ pairs $(\mathbf{h},p)$ with $r(\mathbf{h})=n$ and $r_p(\mathbf{h})=n-1$ as then $p \mid D(\mathbf{h})$ and so $p$ is determined by $\mathbf{h}$ up to $\mathcal{O}(H^{\varepsilon})$ many choices.

\section{A new proof of Davenport's Shrinking Lemma}

In previous work on general cubic forms, a crucial tool for dealing with the bilinear counting problems as seen in \eqref{vdc} as well as more general versions for the Weyl sums was the following result of Davenport, also known as the Shrinking Lemma.

\begin{lemma}[Davenport's Shrinking Lemma]
\label{shrink}
Let $L=(L_1,\dots,L_n) \in \mathbb{R}^{n \times n}$ be a symmetric matrix. Let $P \ge 1$ and $0<Z<1$ be real numbers. Then
\[\#\left\{\mathbf{x} \in \mathbb{Z}^n: \vert \mathbf{x}\vert \le P, \|L_i(\mathbf{x})\| <\frac{1}{2nP} \forall i\right\} \le \left(\frac{4}{Z}\right)^n \cdot \#\left\{\mathbf{x} \in \mathbb{Z}^n: \vert \mathbf{x}\vert \le ZP, \|L_i(\mathbf{x})\| <\frac{Z}{2nP}\forall i\right\}.\]
Here, $\|z\|$ denotes the distance of $z$ to the nearest integer.
\end{lemma}
In only dealing with the Gauß sums in the above discussion we were able to circumvent the use of the lemma, using Lemma \ref{lemma1} and the Pigeonhole Principle as a substitute, but for the Weyl sums it remains an essential ingredient.
Since the only previous proof uses rather intricate tools from the geometry of numbers, it is therefore desirable to present a short and elementary proof which we do in this section.
\begin{proof}[Proof of Lemma \ref{shrink}]
We begin by choosing a prime $q$ such that $\frac{2}{Z} \le q \le \frac{4}{Z}$ which is always possible. Then it will suffice to prove that
\[\#\left\{\mathbf{x} \in \mathbb{Z}^n: \vert \mathbf{x}\vert \le P, \|L_i(\mathbf{x})\| <\frac{1}{2nP} \forall i\right\} \le q^n \cdot \#\left\{\mathbf{x} \in \mathbb{Z}^n: \vert \mathbf{x}\vert \le \frac{2P}{q}, \|L_i(\mathbf{x})\| <\frac{1}{nqP}\forall i\right\}.\]
Denote by $[z]$ the nearest integer to $z$. For each $(\mathbf{a},\mathbf{b}) \in (\mathbb{Z}/q\mathbb{Z})^2$ let
\[N_{\mathbf{a},\mathbf{b}}=\#\left\{\mathbf{x} \in \mathbb{Z}^n: \vert \mathbf{x}\vert \le P, \|L_i(\mathbf{x})\| <\frac{1}{2nP} \forall i, \mathbf{x} \equiv \mathbf{a} \pmod*{q}, ([L_i(\mathbf{x})])_i \equiv \mathbf{b} \pmod*{q}\right\}.\]
Clearly the LHS of our inequality now decomposes as $\sum_{\mathbf{a},\mathbf{b}} N_{\mathbf{a},\mathbf{b}}$. Now observe that if $\mathbf{x}_1$ and $\mathbf{x}_2$ are counted by $N_{\mathbf{a},\mathbf{b}}$, then $\mathbf{x}:=\frac{\mathbf{x}_2-\mathbf{x}_1}{q}$ is counted by the RHS of our inequality. Hence it follows that
\[N_{\mathbf{a},\mathbf{b}} \le \#\left\{\mathbf{x} \in \mathbb{Z}^n: \vert \mathbf{x}\vert \le \frac{2P}{q}, \|L_i(\mathbf{x})\| <\frac{1}{nqP}\forall i\right\}\]
which is already enough to deduce our claimed inequality with a factor of $q^{2n}$ instead of $q^n$, since there are $q^{2n}$ choices of $(\mathbf{a},\mathbf{b})$.

\medskip

To conclude the stronger claim, it will thus suffice to show that $N_{\mathbf{a},\mathbf{b}} \ne 0$ only for at most $q^n$ choices of $(\mathbf{a},\mathbf{b})$.

Indeed, this will follow immediately if we can show that the $2n \times 2n$ matrix with columns $\left(\mathbf{x},[L_i(\mathbf{x})]\right)$ for $\mathbf{x}$ counted by the LHS of our inequality has rank at most $n$.

However, note that by our estimate on $\|L_i(\mathbf{x})\|$ and the symmetry of $L$ we have
\[\mathbf{y} \cdot ([L_i(\mathbf{x})])_i=\mathbf{x} \cdot ([L_i(\mathbf{y})])_i\]
for all $\mathbf{x}$ and $\mathbf{y}$ counted, since both sides are integers and differ by less than $2n \cdot P \cdot \frac{1}{2nP}=1$.

Hence, if we add to our matrix the columns $\left(-[L_i(\mathbf{x})],\mathbf{x}\right)$ each column of the new part will be orthogonal to each column of the old part, and since they both have the same rank, both parts can have rank at most $n$, as desired.
\end{proof}

\section*{Acknowledgments}

This work was carried out while the author was a Ph.D. student at the University of Göttingen, supported by the DFG Research Training Group 2491 \lq Fourier Analysis and Spectral Theory\rq{}. I would like to thank my supervisor Jörg Brüdern for introducing me to the topic and for encouraging me to work on it.

\printbibliography

\end{document}